\documentclass{article}
\usepackage[english]{babel}
\usepackage[latin1]{inputenc}
\usepackage{amsmath, amsfonts, amssymb,amsthm}
\usepackage{mathptmx}
\usepackage{mathtools}
\usepackage[all]{xy}
\usepackage{url}
\usepackage[backref=page]{hyperref}
\hypersetup{pdfpagemode=UseNone, pdfstartview={FitH}}
\usepackage{cleveref}
\newtheorem{theorem}{Theorem}[section]
\newtheorem{remark}{Remark}[section]
\newtheorem{problem}{Problem}[section]
\newtheorem{definition}{Definition}
\newtheorem{proposition}{Proposition}[section]
\newtheorem{lemma}{Lemma}[section] 
\newtheorem{corollary}{Corollary}[section]
\newtheorem{counter}{Counterexample}[section]
\crefname{problem}{Problem}{Problems}
\Crefname{problem}{Problem}{Problems}
\newcommand{\lag}{\left \langle}
\newcommand{\rog}{\right \rangle}
\newcommand{\df}{\mathrel{\mathop:}=}

\begin{document}
\date{}
\title{On some logical and algebraic properties of axiomatic extensions of the monoidal t-norm based logic MTL related with single chain completeness}
\author{Matteo Bianchi\\{\small Department of Mathematics ``Federigo Enriques''}\\{\small Università degli Studi di Milano}\\{\footnotesize\texttt{\href{mailto:matteo.bianchi@unimi.it}{matteo.bianchi@unimi.it}}}}
\maketitle
\begin{abstract}
In \cite{monchain} are studied, for the axiomatic extensions of the monoidal t-norm based logic (\cite{eg}), the properties of single chain completeness.

On the other side, in \cite[Chapter 5]{gjko} are studied many logical and algebraic properties (like Halldén completeness, variable separation properties, amalgamation property etc.), in the context of substructural logics. The aim of this paper is twofold: first of all we will specialize the properties studied in \cite[Chapter 5]{gjko} from the case of substructural logics to the one of extensions of MTL, by obtaining some general characterization. Moreover we will show that some of these properties are indeed strictly connected to the topics developed in \cite{monchain}. This will help to have a better intuition concerning some open problems of \cite{monchain}.
\end{abstract}
\section{Introduction}
Monoidal t-norm based logic (MTL) was introduced in \cite{eg} as the basis of a logical framework of many-valued logics initially introduced by Petr Hájek in \cite{haj}.

MTL and its extensions (logics obtained from it by adding other axioms) are all algebraizable in the sense of \cite{bp} and their corresponding classes of algebras form an algebraic variety (see \cite{nog,dist}). Given an axiomatic extension L of MTL one can study the completeness properties of L with respect to some classes of L-algebras: for example the class of all L-algebras, the one of all L-chains, the one of t-norm based L-algebras (if any). As shown in \cite{haj,eg} every extension L of MTL is strongly complete w.r.t. the class of all L-chains. Why is it important to find completeness results, for a logic, with respect to a class of totally ordered algebras? As pointed out by Petr Hájek in the introduction of his book \cite{haj} one of the desirable characteristic of his framework of many-valued logics is the comparative notion of truth: that is, sentences may be compared according to their truth values. So, if we agree with this point of view, then we must necessarily deal with totally ordered algebras, since the ``truth values'' must be comparable. However, the class of all L-chains is very large and we do not have a ``candidate" algebraic structure in which we can evaluate the truth values of formulas.

This problem can be overcome when the logic is complete with respect to a totally ordered algebra: in this case we say that this logic is single chain complete. The article \cite{monchain} presents a systematic study of completeness properties of this type, for the axiomatic extensions of MTL: many problems, however, remain open. 

The book \cite{gjko}, instead, is a reference monograph concerning residuated lattices as well as the associated substructural logics: in Chapter 5 of this book, many logical and algebraic properties are studied (like disjunction properties, Halldén completeness, deductive Maksimova variable separation properties, pseudo-relevance properties, amalgamation and interpolation properties), by showing many interesting equivalent characterizations of them. 

In this paper we specialize many properties of \cite[Chapter 5]{gjko} to the case of MTL logic and its extensions, by showing that most of them are indeed connected with the single chain completeness results of \cite{monchain}. We will conclude the paper by discussing some open problems.
\section{Preliminaries}
\subsection{Syntax}
Monoidal t-norm based logic (MTL) was introduced in \cite{eg}: it is based over connectives $\{\&, \land, \to, \bot\}$ (the first three are binary, whilst the last one is $0$-ary), and a denumerable set of variables. The notion of formula is defined inductively in the usual way. 

Useful derived connectives are the following
\begin{align}
\tag{negation}\neg\varphi\df& \varphi\to\bot\\
\tag{disjunction}\varphi\vee\psi\df& ((\varphi\to\psi)\to\psi)\land((\psi\to\varphi)\to\varphi)
\end{align}
For reader's convenience we list the axioms of MTL
\begin{align}
\tag{A1}&(\varphi \rightarrow \psi)\rightarrow ((\psi\rightarrow \chi)\rightarrow(\varphi\rightarrow \chi))\\
\tag{A2}&(\varphi\&\psi)\rightarrow \varphi\\
\tag{A3}&(\varphi\&\psi)\rightarrow(\psi\&\varphi)\\
\tag{A4}&(\varphi\land\psi)\rightarrow \varphi\\
\tag{A5}&(\varphi\land\psi)\rightarrow(\psi\land\varphi)\\
\tag{A6}&(\varphi\&(\varphi\rightarrow \psi))\rightarrow (\psi\land\varphi)\\
\tag{A7a}&(\varphi\rightarrow(\psi\rightarrow\chi))\rightarrow((\varphi\&\psi)\rightarrow \chi)\\
\tag{A7b}&((\varphi\&\psi)\rightarrow \chi)\rightarrow(\varphi\rightarrow(\psi\rightarrow\chi))\\
\tag{A8}&((\varphi\rightarrow\psi)\rightarrow\chi)\rightarrow(((\psi\rightarrow\varphi)\rightarrow\chi)\rightarrow\chi)\\
\tag{A9}&\bot\rightarrow\varphi
\end{align}
As inference rule we have modus ponens:
\begin{equation}
\tag{MP}\frac{\varphi\quad \varphi\rightarrow\psi}{\psi}
\end{equation}
A logic L is called axiomatic extension of MTL if it is obtained from this last one by adding other axioms. In particular MTL is a substructural logic and also an axiomatic extension of the logic FL$_{ew}$ (\cite{gjko,on}): indeed MTL can also be axiomatized as FL$_{ew}$ plus
\begin{equation}
\tag{prelin}(\varphi\to\psi)\vee(\psi\to\varphi).
\end{equation} 
The notions of theory, syntactic consequence, proof are defined as usual.

Let L be an axiomatic extension of MTL: for a positive integer $n$, L is called $n$-contractive whenever $\vdash_L \varphi^n\to\varphi^{n+1}$.

\noindent For the axiomatic extensions of MTL it holds the following form of deduction theorem: 
\begin{theorem}[\cite{ci}]\label{teo:ldt}
Let L be an axiomatic extension of MTL and $\Gamma,\varphi,\psi$ be a theory and two formulas. It holds that
\begin{equation*}
\Gamma\cup\{\psi\}\vdash_{L}\varphi\quad\text{iff there exists }n\in\mathbb{N}^+\text{ s.t.}\quad \Gamma\vdash_{L}\psi^n\to\varphi.
\end{equation*}
\end{theorem}
It is called local deduction theorem, since $n$ depends on the theory and formulas considered.

\noindent For every $n$-contractive axiomatic extension of MTL we obtain the following (global) form.
\begin{theorem}[{\cite[Theorem 3.3]{hnp}}]\label{teo:conldt}
Let $L, \Gamma,\varphi,\psi$ be an $n$-contractive extension of MTL, a theory and two formulas. It holds that
\begin{equation*}
\Gamma\cup\{\psi\}\vdash_\text{L}\varphi\quad\text{iff}\quad \Gamma\vdash_\text{L}\psi^{n}\to\varphi.
\end{equation*}
\end{theorem}
\subsection{Semantics}\label{subsec:sem}
An FL$_{ew}$-algebra is an algebra $\lag A,*,\Rightarrow,\sqcap,\sqcup,0,1\rog$ such that
\begin{enumerate}
\item $\lag A,\sqcap,\sqcup, 0,1\rog$ is a bounded lattice with minimum $0$ and maximum $1$.
\item $\lag A,*,1 \rog$ is a commutative monoid.
\item $\lag *,\Rightarrow \rog$ forms a \emph{residuated pair}: $z*x\leq y$ iff $z\leq x\Rightarrow y$ for all $x,y,z\in A$.
\footnote{Here the partial order $\leq$ is defined as $x\leq y$ iff $x\sqcap y=x$, for $x,y\in A$.}
\end{enumerate}
An MTL-algebra is an FL$_{ew}$-algebra satisfying
\begin{equation}
\tag{Prelinearity}(x\Rightarrow y)\sqcup(y\Rightarrow x)=1
\end{equation}
Finally, a totally ordered MTL-algebra is called MTL-chain.
\paragraph*{}
The notion of assignment, model and satisfiability are defined as usual: we refer to \cite{eg} for details. 

Finally, if L is an (axiomatic) extension of MTL, with $\mathbb{L}$ we will denote its corresponding variety of algebras.
\section{Algebraic and logical properties of extensions of MTL and single chain completeness results}
We begin with some definitions of properties introduced in \cite{monchain}.
\begin{definition}
Let L be an axiomatic extension of MTL. Then
\begin{itemize}
\item L enjoys the \emph{single chain completeness} (SCC) if there is an L-chain such that L is complete w.r.t. it.
\item L enjoys the \emph{strong single chain completeness} (SSCC) if there is an L-chain such that L is strongly complete w.r.t. it.
\end{itemize}
\end{definition}
\begin{remark}
The reader could note that we do not have defined the notion of finite strong single chain completeness. This is because in \cite[Theorem 3]{monchain} it is shown that this property is equivalent to single chain completeness. Hence we will deal only with this last one. 
\end{remark}
Clearly the SSCC implies the SCC. The vice-versa is left, in \cite{monchain}, as an open problem. In this section we will present some properties that are related to the ones of single chain completeness.
\begin{definition}
We say that a logic L has the \emph{disjunction property} (DP) if $\vdash_L\varphi\vee\psi$ implies that $\vdash_L\varphi$ or $\vdash_L\psi$.
\end{definition}
For example the intuitionistic logic enjoys this property: however it fails for many superintuitionistic logics (see \cite{cz} for a survey) and for classical logic (for this last one $x\vee\neg x$ is a counterexample).

For the case of axiomatic extensions of MTL, we obtain a negative result:
\begin{theorem}
Let L be a (consistent) axiomatic extension of MTL: then DP fails for L.
\end{theorem}
\begin{proof}
The formula $(x\to y)\vee(y \to x)$ is a theorem of L. Consider now the direct product $\mathbf{2}\times\mathbf{2}$ of two copies of two elements boolean algebra: clearly this algebra belongs to the variety of L-algebras. By taking a $\mathbf{2}\times\mathbf{2}$-evaluation $v$ such that $v(x)=\lag0,1\rog$ and $v(y)=\lag1,0\rog$, we obtain $v((x \to y)\vee(y\to x))=1$, whilst $v(x\to y)<1$ and $v(y \to x)<1$. From completeness theorem (\cite{eg}) we have $\not\vdash_L x \to y$, $\not\vdash_L y \to x$.
\end{proof}
There is a property weaker than DP: the Halldén completeness.
\begin{definition}
A logic L has the \emph{Halldén completeness} (HC) if for every formulas $\varphi,\psi$ with no variables in common, $\vdash_L\varphi\vee\psi$ implies that $\vdash_L\varphi$ or $\vdash_L\psi$.
\end{definition}
There is an interesting algebraic characterization of HC, for the extensions of FL$_{ew}$
\begin{definition}\label{def:wc}
An FL$_{ew}$-algebra is said to be \emph{well-connected} whenever for every pair of elements $x,y$, if $x\sqcup y=1$, then $x=1$ or $y=1$. 
\end{definition}
\begin{theorem}[{\cite[Theorem 5.28]{gjko}}]\label{teo:subhc}
Let L be a logic over FL$_{ew}$. The following are equivalent:
\begin{enumerate}
\item L has the Halldén completeness.
\item There is a well-connected FL$_{ew}$-algebra $\mathcal{A}$ such that L is complete w.r.t. it.
\item L is meet irreducible (in the lattice of axiomatic extensions of FL$_{ew}$).
\end{enumerate}
\end{theorem}
Moving to the hierarchy of MTL and its extensions, as shown in \cite[Corollaries 4.19, 4.20]{nog}, we have
\begin{proposition}
An MTL-algebra is well-connected if and only if it is a chain.
\end{proposition}
We can reformulate \Cref{teo:subhc} as follows
\begin{theorem}\label{teo:mtlhc}
Let L be an axiomatic extension of MTL. The following are equivalent
\begin{enumerate}
\item L has the Halldén completeness.
\item There is an MTL-chain $\mathcal{A}$ such that L is complete w.r.t. it.
\item L is meet irreducible (in the lattice of axiomatic extensions of MTL).
\end{enumerate}
\end{theorem}
Hence:
\begin{corollary}\label{cor:scchc}
For every axiomatic extension of MTL, the HC is equivalent to the SCC.
\end{corollary}
Moreover
\begin{theorem}[{\cite[Corollary 5.30]{gjko}}]\label{teo:hccon}
Let L be an $n$-contractive substructural logic over FL$_{ew}$: the following are equivalent.
\begin{itemize}
\item L enjoys the HC.
\item There is a subdirectly irreducible L-algebra such that L is complete w.r.t. it.
\end{itemize}
\end{theorem}
A particular case of single-chain completeness is the following:
\begin{definition}
Let L be an axiomatic extension of MTL. We say that L enjoys the subdirect single chain completeness (subSCC) if there is a generic subdirectly irreducible L-algebra.
\end{definition}
Clearly the subSCC implies the SCC, since every subdirectly irreducible MTL-algebra is totally ordered: for the $n$-contractive extensions of MTL, thanks to \Cref{cor:scchc} and  \Cref{teo:hccon}, also the converse holds.
\begin{theorem}\label{teo:subscc}
Let L be an $n$-contractive extension of MTL: the following are equivalent.
\begin{itemize} 
\item L enjoys the SCC. 
\item L enjoys the subSCC.
\end{itemize}
\end{theorem}
\begin{corollary}
The following $n$-contractive extensions of MTL enjoy the subSCC: WNM, RDP, NM, G, {\L}$_n$, SMTL$^n$, SBL$^n$ (see \cite{eg,rdp,haj,grig,bln} for their axiomatization). 
\end{corollary}
\begin{proof}
In \cite{nog,eg,haj,grig,monchain,bln,rdp} are shown examples of generic chains for the varieties associated to these logics. The claim of the corollary follows from \Cref{teo:subscc}.
\end{proof}
Since in \cite[Proposition 37]{wcmtl} it is shown that every locally finite subvariety of MTL-algebras is $n$-contractive, for some $n$, then from the previous theorem we have:
\begin{corollary}
Let L be an extension of MTL whose corresponding variety is locally finite: the following are equivalent.
\begin{itemize} 
\item L enjoys the SCC. 
\item L enjoys the subSCC.
\end{itemize}
\end{corollary}

\begin{problem}\label{prob:1}
Are there (non $n$-contractive) axiomatic extensions of MTL enjoying the SCC but not the subSCC ?
\end{problem}
Even if it is not a solution for the previous problem, we have the following result:
\begin{theorem}
The following non $n$-contractive extensions of MTL enjoy the subSCC: SMTL, BL, SBL, \L, $\Pi$ (see \cite{rat,haj,cegt} for their axiomatization).
\end{theorem}
\begin{proof}
First of all, note that all these logics enjoys the SCC (for details see \cite{monchain}).

As regards to $\text{L}\in\{\text{SMTL}, \text{BL}, \text{SBL}\}$ take a generic L-chain $\mathcal{A}$: then $\mathcal{A}\oplus\mathbf{2}$ is a subdirectly irreducible generic L-chain (indeed, as pointed out in \cite{nog} the varieties corresponding to these logics are closed under ordinal sums). Finally, concerning $\text{L}\in\{\text{\L}, \Pi\}$, note that the standard L-algebra is a subdirectly irreducible generic L-chain.
\end{proof}
Another property, similar to the HC, is the following
\begin{definition}
A logic L has the \emph{deductive Maksimova's variable separation property} (DMVP), if for all sets of formulas
$\Gamma\cup \{\varphi\}$ and $\Sigma \cup \{\psi\}$ that have no variables in common, $\Gamma, \Sigma \vdash_L \varphi\vee\psi$ implies $\Gamma\vdash_L \varphi$ or $\Sigma \vdash_L \psi$.
\end{definition}
As can be easily seen, the DMVP implies the HC. Moreover, the first property can be algebraically characterized as follows:
\begin{theorem}[{\cite[Theorem 6.9]{ki}}]\label{teo:dmvp}
The following conditions are equivalent for every substructural logic L over FL$_{ew}$:
\begin{itemize}
\item L has the DMVP.
\item All pairs of subdirectly irreducible L-algebras are jointly embeddable into a well-connected L-algebra.
\item All pairs of subdirectly irreducible L-algebras are jointly embeddable into a subdirectly irreducible L-algebra.
\end{itemize}
\end{theorem}
\begin{problem}\label{prob:2}
Are there some examples of extensions of MTL enjoying the HC but not the DMVP ?
\end{problem}
Consider the following property:
\begin{definition}
Let L be an axiomatic extension of MTL. We say that its corresponding variety enjoys the chain joint embedding property (CJEP) whenever every pair of L-chains is embeddable into some L-chain.
\end{definition}
The CJEP is very important for the (strong) single chain completeness results:
\begin{theorem}[\cite{monchain}]
Let L be an axiomatic extension of MTL. Then L enjoys the SSCC iff its corresponding variety has the CJEP.
\end{theorem}
Moreover, as a consequence of \Cref{teo:dmvp}, we have the following result:
\begin{theorem}
Let L be an axiomatic extension of MTL. If the variety of L-algebras enjoys the CJEP, then L has the DMVP. 
\end{theorem}
\begin{problem}\label{prob:dmvp}
Does the DMVP imply the CJEP ?
\end{problem}
Consider now:
\begin{definition}
\begin{itemize}
\item A logic L has the \emph{pseudo-relevance property} (PRP), if for all pairs of formulas $\varphi, \psi$ with no variables in common, \mbox{$\vdash_L\varphi\to\psi$} implies either $\vdash_L\neg\varphi$ or $\vdash_L\psi$.
\item A logic L has the \emph{deductive pseudo-relevance property} (DPRP), if for every theory $\Gamma$ and formula $\psi$ with no variables in common, \mbox{$\Gamma\vdash_L\psi$} implies either $\Gamma\vdash_L\bot$ or $\vdash_L\psi$.
\item A logic L has the \emph{strong deductive pseudo-relevance property} (SDPRP), if for every sets of formulas $\Gamma$ and  $\Sigma\cup\{\psi\}$ with no variables in common, \mbox{$\Gamma, \Sigma\vdash_L\psi$} implies either $\Gamma\vdash_L\bot$ or $\Sigma\vdash_L\psi$.
\end{itemize}
\end{definition}
It holds that
\begin{theorem}[\cite{gjko}]\label{teo:sdprp}
Let L be a logic over FL$_{ew}$.
\begin{itemize}
\item L enjoys the SDPRP if and only if every pair of subdirectly irreducible L-algebras is jointly embeddable into an L-algebra.
\item SDPRP implies DPRP for every L, and the converse holds also when the variety of L-algebras has the CEP (i.e. every pair of L-algebras $\mathcal{A}, \mathcal{B}$, with $\mathcal{A}$ being a subalgebra of $\mathcal{B}$, is such that for every congruence $\theta$ of $\mathcal{A}$ there is a congruence $\theta'$ of $\mathcal{B}$ such that $\theta= \theta'\cap\mathcal{A}^2$).
\end{itemize}
\end{theorem}
Now, since every variety of MTL-algebras enjoys the CEP (\cite[page 42]{nog}) then we have
\begin{theorem}
For every variety of MTL-algebras the SDPRP is equivalent to the DPRP. 
\end{theorem}
One can ask which is the relation between the DPRP and PRP: a first result, shown in \cite[page 45]{ki} in the context of substructural logics, is the following.
\begin{theorem}[{\cite[page 45]{ki}}]\label{teo:prp}
For every axiomatic extension of MTL, the PRP implies the DPRP.
\end{theorem}
Moreover
\begin{theorem}[{\cite[Theorem 5.57]{gjko}}]
Every extension of the logic FL$_{ew}$ with the axiom $\neg(\varphi\land\neg\varphi)$ has the PRP.
\end{theorem}
Hence
\begin{corollary}\label{cor:prp}
Every axiomatic extension of SMTL enjoys the PRP and the SDPRP.
\end{corollary}
This result can be strengthened. 
\begin{lemma}\label{lem:mtl}
The formula $(\varphi\land\neg\varphi)\to(\psi\vee\neg\psi)$ is a theorem of MTL.
\end{lemma}
\begin{proof}
Suppose not: it follows that there is an MTL-chain with two elements $x,y$ such that $\min(x,\sim x)>\max(y,\sim y)$. However this is a contradiction, since $x>y$ implies that $\sim x \leq \sim y$.
\end{proof}
\begin{theorem}
Let L be an axiomatic extension of MTL: then L enjoys the PRP if and only if it is an extension of SMTL.
\end{theorem}
\begin{proof}
If L is an extension of SMTL, then the result follows from \Cref{cor:prp}. 

Suppose now that L is not an extension of SMTL: it follows that $\not\vdash_L\neg(\varphi\land\neg\varphi)$ and hence there is an L-chain $\mathcal{A}$ with an element $a>0$ such that $\sim a>0$ (i.e. $a$ is a non trivial zero divisor). Take $\varphi\df(x\land\neg x)$ and $\psi\df(y\vee\neg y)$: thanks to \Cref{lem:mtl} we have that $\vdash_L\varphi\to\psi$, whilst $\not\models_\mathcal{A}\neg\varphi,\, \not\models_\mathcal{A}\psi$, and hence $\not\vdash_L\neg\varphi,\, \not\vdash_L\psi$. Hence L does not have the PRP.
\end{proof}
Moreover, observe that the converse of \Cref{teo:prp} does not hold, in general.
\begin{counter}\label{counter:prp}
Consider \L ukasiewicz logic (\L) and take $\varphi\df x\land \neg x,\, \psi\df y\vee \neg y$. Clearly $\vdash_{\text{\L}}\varphi\to\psi$, but $\not\vdash_{\text{\L}}\neg\varphi,\, \not\vdash_{\text{\L}}\psi$ (this can be easily checked over the standard MV-algebra). However \L ukasiewicz logic enjoys the DPRP, since the variety of MV-algebras enjoys the CJEP (see \cite{monchain}): indeed $\varphi\vdash_{\text{\L}}\psi$ and $\varphi\vdash_{\text{\L}}\bot$.

\noindent Note that this counterexample also applies to Nilpotent Minimum logic (\cite{eg}). 
\end{counter}
As we have pointed out in \cref{counter:prp}, the CJEP implies the DPRP. The following theorem shows the relation between CJEP and the other properties:
\begin{theorem}
Let L be an extension of MTL. 
\begin{itemize}
\item If the variety of L-algebras has the CJEP then L enjoys the SSCC, HC, DMVP, SDPRP. 
\item If L enjoys the DMVP then L enjoys also the DPRP.
\end{itemize}
\end{theorem}
\begin{proof}
An easy check from the previous results.
\end{proof}
We now introduce another algebraic property.
\begin{definition}
We say that a variety $K$ of MTL-algebras has the \emph{amalgamation property} (AP) if for every tuple $\lag \mathcal{A}, \mathcal{B}, \mathcal{C}, i, j\rog$, where $\mathcal{A}, \mathcal{B}, \mathcal{C}\in K$ and $\mathcal{A}\xhookrightarrow{i} \mathcal{B}$, $\mathcal{A}\xhookrightarrow{j} \mathcal{C}$, there is a tuple $\lag \mathcal{D},h,k\rog$, with $\mathcal{D}\in K$, $\mathcal{B}\xhookrightarrow{h} \mathcal{D}$, $\mathcal{C}\xhookrightarrow{k} \mathcal{D}$, such that $h\circ i=k\circ j$.
\end{definition}
An easy check shows that the AP implies the DPRP, thanks to \Cref{teo:sdprp}.
 
There is, moreover, a logical property that is strictly connected to the AP.
\begin{definition}\label{def:dip}
A logic L has the \emph{deductive interpolation property} (DIP) if for any theory $\Gamma$ and for any formula $\psi$ of L, if $\Gamma \vdash _{L}\psi $, then there is a formula $\gamma $ such that $\Gamma \vdash _{L}\gamma $, $\gamma \vdash_{L}\psi $ and every propositional variable occurring in $\gamma $ occurs both in $\Gamma$ and in $\psi $.
\end{definition}
As shown in \cite{gjko} (see also \cite[Theorem 5.8]{go}) DIP and AP are equivalent:
\begin{theorem}[\cite{gjko}]\label{teo:nmdip}
An axiomatic extension of MTL enjoys the DIP iff the corresponding variety has the AP.
\end{theorem}
In \Cref{fig:1} are summarized the connections between the various properties hitherto introduced. The ``negative'' arrows follow from \cref{counter:prp} and some results pointed out \cite{monchain}, concerning AP, CJEP and SCC: indeed in \cite{mo06,monchain} it is shown that: 
\begin{itemize}
\item There are subvarieties of BL-algebras enjoying the AP, but for which the SCC does not hold: for example, the join of the varieties of Gödel and Product algebras.
\item There are subvarieties of BL-algebras enjoying the CJEP (and hence the SSCC and SCC), but for which the AP fails to hold: for example, every variety generated by a finite Gödel-chain with more than three elements.
\end{itemize}
These results clarify immediately the relations between AP and SCC, AP and SDPRP, SDPRP and CJEP, DPRP and PRP.
As a consequence we obtain that the SDPRP does not imply the DMVP: if this was true, then by following the arrows in the diagram we would have that the AP implies the SCC: a contradiction.

We now discuss the relation between PRP and SSCC. First of all, the second one does not imply the first-one: \L ukasiewicz logic is a counterexample. Finally, the negative arrow from PRP to SSCC is a consequence of the fact that there are some extensions of SMTL whose corresponding variety does not enjoy the SSCC: for example, as previously noticed, the logic associated to the join of the varieties of Gödel and product algebras, that clearly is an extension of SMTL.
\begin{figure}[h] 
\hspace{1.5 cm}
\xymatrix@C=2.5pc@R=2.5pc{
SCC \ar@<0.5ex>[dr] |\setminus \ar@{<->}[r] & HC \ar@<1ex>[r]^(.4){?}                                    & \ar[l] DMVP \ar[d] \ar@<1ex>[r]^?                   & \ar[l] CJEP \ar@<0.65ex>[dl]\ar@{<->}[r] & SSCC \ar@<0.5ex>[d] |\setminus \\
DIP \ar@{<->}[r] & AP \ar@<1ex>[r] \ar@<0.5ex>[ul] |\setminus                           &        \ar[l] |(.6){\setminus} \ar@<1ex>[u] |\setminus SDPRP \ar@<0.35ex>[ur] |\setminus        \ar@{<->}[r]        &   DPRP\ar@<1ex>[r] |\setminus                               & \ar[l] PRP \ar@<0.5ex>[u] |\setminus
}
\caption{Relations between the properties previously introduced, for axiomatic extensions of MTL.}\label{fig:1}
\end{figure}
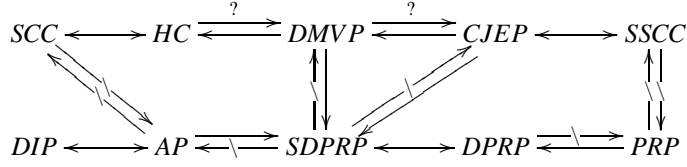
\noindent The last property that we want to discuss is Craig interpolation:
\begin{definition}
Let L be an axiomatic extension of MTL. We say that the \emph{%
Craig interpolation theorem} holds for L iff for any two formulas $\varphi $ and $\psi $
of L, if $\vdash _{L}\varphi \rightarrow \psi $, then there is a formula $%
\gamma $ such that $\vdash _{L}\varphi \rightarrow \gamma $, $\vdash _{L}\gamma
\rightarrow \psi $ and every propositional variable occurring in $\gamma$
occurs both in $\varphi$ and in $\psi$.
\end{definition}
This property, however, fails for many axiomatic extensions of MTL: indeed in \cite{mo06} it is shown that this property holds only for G, G$_3$ and classical logic, among the axiomatic extensions of BL. 

Nevertheless, for the $n$-contractive axiomatic extensions of MTL that enjoys the DIP, we can obtain a weaker form of Craig's theorem (a generalization of the theorem given in \cite{bln} for some families of $n$-contractive extensions of BL).
\begin{theorem}[Weak Craig interpolation theorem]
Let L be an $n$-contractive extension of MTL that enjoys the DIP. For every pair of formulas $\varphi, \psi$, if $\vdash_\text{L}\varphi^{n}\rightarrow \psi $, then there is a formula $\gamma $ such that $\vdash_\text{L}\varphi ^{n}\rightarrow \gamma $, $\vdash_\text{L}\gamma^{n} \rightarrow \psi $ and every propositional variable occurring in $\gamma $ occurs both in $\varphi $ and in $\psi $.
\end{theorem}
\begin{proof}
An easy consequence of \Cref{teo:conldt} and  \Cref{def:dip}.
\end{proof}
\section{Conclusions and discussion of the open problems}
One of the main problems (in the propositional case) left open in \cite{monchain} was:
\begin{problem}\label{prob:3}
Let L be an axiomatic extension of MTL enjoying the SCC: does L enjoy the SSCC ?
\end{problem}
Even if we do not have solved it, note that this problem is connected with \cref{prob:2} and \cref{prob:dmvp} (the ``interrogative'' arrows in \Cref{fig:1}). Indeed if a logic enjoys the HC but not the DMVP, then the SCC holds, but the SSCC fails; the same if a logic enjoys the DMVP, but not the CJEP. Hence, a negative answer to one of these two problems will necessary involve a negative answer to \cref{prob:3}.

As explained in \cite[page 163]{monchain} another open problem concerns the SCC for MTL, IMTL, $\Pi$MTL: this is still unsolved, but our \Cref{teo:mtlhc} could help to get an intuition towards a solution, since it provides some equivalent characterizations for the SCC. 

Future directions of research will concern these open problems, but also the first-order case. Indeed, in this paper we have completely overlooked the properties of SCC and SSCC for the first-order extensions of MTL: this has been done not because these problems are poorly relevant, but because the situation is much more complicate than in the propositional case. For example in \cite{monchain}, differently from the propositional case, for the SSCC in the first-order case it has been found only a sufficient (and not necessary) condition. Moreover there are many extensions of MTL enjoying the SSCC in the propositional case, but not in the first-order case.
\bibliography{MTLlogalg}
\bibliographystyle{amsalpha}
\end{document}